\documentclass[12pt, a4paper]{article}
\usepackage{amsfonts,amssymb,amsmath,amscd,latexsym,makeidx,theorem}

\usepackage{upref}
\usepackage{xcolor}
\usepackage{txfonts}
\usepackage{mathtools}
\mathtoolsset{showonlyrefs}

\usepackage{hyperref} 

\newtheorem{thm}{Theorem}[section]
\newtheorem{thmx}{Theorem}

\newtheorem{lem}[thm]{Lemma}
\newtheorem{prop}[thm]{Proposition}

\newtheorem{pb}{Open Question}

\newtheorem{defn}{Definition}[section]

\newenvironment{proof}{\noindent\emph{Proof.}}{\hfill$\square$\medskip}

\newcommand{\supp}{\mathrm{supp}}

\newcommand{\M}[1]{\mathcal{#1}}

\newcommand{\sph}{\mathbb{S}}
\newcommand{\de}{\partial}

\newcommand{\R}{\mathbb{R}}
\newcommand{\ve}{\varepsilon}

\newcommand{\cl}[1]{\overline{#1}}

\renewcommand{\(}{\left(}
\renewcommand{\)}{\right)}

\author{Ali Hyder\thanks{The author  acknowledges the support of the Department of Atomic Energy, Government of India, under Project Identification No. RTI 4014} 
\\ {\small Tata Institute of Fundamental Research,}\\{\small  Centre for Applicable Mathematics}\\ {\small \texttt{hyder@tifrbng.res.in}}\and  Luca Martinazzi \\{\small Dipartimento di Matematica}\\  {\small Sapienza, Universit\`a di Roma} \\ {\small \texttt{luca.martinazzi@uniroma1.it}}}

\title{Non-constancy and multiplicity of half-harmonic maps from intervals into the circle}

\begin{document}

\maketitle

\begin{abstract}
We study one-dimensional half-harmonic maps from the real line into the circle with prescribed exterior data. We show that, for every positive integer $k$, if two disjoint intervals are sufficiently close, there exist at least $k$ distinct non-constant half-harmonic maps with constant exterior data. More generally, we establish a multiplicity result for boundary data with energy below the critical threshold $2\pi$ by introducing a local relative degree and proving corresponding degree-jump estimates.

Working on a single interval and for boundary data arising as traces of finite Blaschke products, we investigate the existence and non-existence of energy minimizers in prescribed degree classes. 
\end{abstract}

\section{Introduction}
 
 Consider the space $\dot H^{\frac12}(\R, \mathbb{C})$ of measurable maps $u:\R\to \mathbb{C}$ such that 
$$[u]_{H^{\frac{1}{2}}}^2:= \int_{\R}\int_{\R}\frac{|u(x)-u(\xi)|^2}{(x-\xi)^2}dxd\xi<\infty,$$
and its subset of $\sph^1$-valued maps
$$\dot H^{\frac12}(\R, \sph^1) = \left\{u\in \dot H^{\frac12}(\R, \mathbb{C}): |u(x)|=1 \text{ for a.e. }x\in\R\right\},$$
endowed with the same seminorm. 
Following \cite{DLR}, we consider the fractional energy
\begin{equation}\label{12energy}
E(u):=\int_{\R}\left|(-\Delta)^\frac14u\right|^2dx, \quad u\in \dot H^{\frac12}(\R, \sph^1).
\end{equation}
It is well-known  (see \cite{HM} for details) that
\begin{equation}\label{12energybis}
E(u)=\|(-\Delta)^\frac14 u\|_{L^2}^2=\frac{1}{2\pi}\int_{\R}\int_{\R}\frac{|u(x)-u(y)|^2}{|x-y|^2}dxdy = \int_{\R^2_+} |\nabla \tilde u|^2 dx dy,\quad u\in \dot H^{\frac12}(\R, \sph^1),
\end{equation}
where $\tilde u\in \dot H^1(\R^2_+)$ is the finite-energy harmonic extension of $u$, given by the usual Poisson formula:
$$\tilde u(x,y)=\frac{1}{\pi} \int_\R\frac{yu(\xi)}{y^2+(x-\xi)^2}d\xi,\quad (x,y)\in \R^2_+.$$
Let $I=(-1,1)$ and for a given $g\in \dot H^{\frac12}(\R, \sph^1)$ set
$$\mathcal{E}_g= \{u\in \dot H^{\frac12}(\R, \sph^1):u=g\text{ in }\R\setminus I\}.$$
A minimizer $u_0$ of $E$ in $\mathcal{E}_g$ exists by direct methods and is a half-harmonic map in $I$ (equivalently, a critical point of $E$ in $\mathcal{E}_g$, see e.g. \cite{milsir}), namely it solves
$$(-\Delta)^\frac12 u(x):=-\frac{\partial \tilde u(x,0)}{\partial y}\perp T_{u(x)}\sph^1,$$
hence it is smooth in $I$ by a result of Da Lio and Rivi\`ere \cite{DLR} (see also \cite{MazSch0,MazSch2,milpeg,milpegsch,milsir, Sch2} for more recent approaches). In our previous work \cite{HM} we studied its non-uniqueness, namely the existence of other half-harmonic maps in $\mathcal{E}_g$, in analogy with similar results of Brezis-Coron \cite{bc} and Jost \cite{jost} regarding harmonic maps from the disk $D^2$ into the sphere $\sph^2$. In particular we proved:

\begin{thmx}[\cite{HM}]\label{thmx1} Let $u_0$ be a minimizer of $E$ in $\mathcal{E}_g$ for some $g\in \dot H^{\frac12}(\R,\sph^1)$. If $g$ is non-constant in $\R\setminus I$, there exists a second critical point $\bar u\ne u_0$ of $E$ in $\mathcal{E}_g$.
\end{thmx}

The idea behind Theorem \ref{thmx1} is to minimize $E$ in a homotopy class different from the one of $u_0$. More precisely define $\deg(u)$ for $u\in \dot H^{\frac12}(\R, \sph^1)$, as $\deg(u\circ\pi)$, where $\pi:\sph^1\setminus\{i\}\to \R$ is the stereographic projection {from the North pole}.
{An explicit definition is given by
\begin{equation}\label{defdeg1}
\deg(u)=\frac{1}{2\pi i}\int_\R \mathcal{H}[(-\Delta)^\frac14 \bar u] (-\Delta)^\frac14 u dx,\quad \text{for }u\in \dot H^{\frac12}(\R, \sph^1),\end{equation}
where $\mathcal{H}$ is the Hilbert transform, see \cite[Section 2]{HM}.}

Now define
$$\M{E}_{g,d}=\{u\in \dot H^{\frac12}(\R,\sph^1): u=g \text{ in }{\R\setminus I},\, \deg(u)=d\}.$$
Then Theorem \ref{thmx1} follows from proving the existence of minimizers of $E$ restricted to ${\M{E}_{g,d_0+ 1}}$ or to ${\M{E}_{g,d_0- 1}}$ (depending on the orientation of $u_0$), where $d_0=\deg(u_0)$.

\subsection{Constancy and non-constancy results}

The condition $g\not\equiv const$ in $\R\setminus I$ in Theorem \ref{thmx1} turns out to be necessary as the authors proved, for a general target closed manifold:

\begin{thmx}[\cite{HM}]\label{thmx2} Let  $\mathcal{N}$ be a smooth closed manifold isometrically embedded in $\R^N$.  Let $$u\in \dot H^{\frac12}(\R,\mathcal{N}):=\left\{u\in \dot H^{\frac12}(\R, \R^N): u(x)\in \mathcal{N} \text{ for a.e. }x\in\R\right\},$$
be half-harmonic in $I$, i.e.
\begin{equation}\label{eq12harmN}
(-\Delta)^\frac{1}{2} u(x):=-\frac{\de \tilde u(x,0)}{\de y} \perp T_{u(x)} \mathcal{N} \quad \text{for }x\in I,
\end{equation}
where $\tilde u:\R^2_+\to \R^N$ is the Poisson harmonic extension. If $u(x)\equiv P$ in $\R\setminus I$ for some $P\in \mathcal{N}$    then $u\equiv P$ in $\R$.
\end{thmx}

This result is the fractional counterpart of the following Lemaire's constancy theorem:

\begin{thmx}[\cite{lem}]\label{thmx3} Given a smooth contractible surface $M$ with boundary and $\mathcal{N}$ a smooth Riemannian manifold without boundary, any harmonic map $u:M\to \mathcal{N}$ such that $u\equiv P$ on $\partial M$ for some $P\in \mathcal{N}$ is constant.
\end{thmx}

The crucial step in the proofs of both Theorems \ref{thmx2} and \ref{thmx3} is proving conformality: that of the Poisson harmonic extension $\tilde u$ in Theorem \ref{thmx2}, and that of the map $u$ itself in Theorem \ref{thmx3}. To emphasise similarities and differences between the two results, we recall the main ideas behind them.

In the case of Theorem \ref{thmx3}, instead of the original proof, we sketch an approach based on the Hopf differential\footnote{{while we do not have a precise reference for this approach, we think that it is based on widely known ideas.}}. Up to a conformal diffeomorphism, consider $M=D^2=\{z\in \mathbb{C}:|z|\le 1\}$ with coordinates $z=x+iy$ and set 
\begin{equation}\label{defH}
H(z)=\left(|u_x|^2_\mathcal{N}-|u_y|^2_\mathcal{N}\right)-2i\langle u_x, u_y\rangle_\mathcal{N},
\end{equation}
which is holomorphic because $u$ is harmonic.
In polar coordinates one computes
$$H_1(z):=z^2H(z) = r^2|u_r|^2_\mathcal{N}-|u_\theta|^2_\mathcal{N} -2ir\langle u_r,u_\theta\rangle_\mathcal{N}.$$
On $\partial D^2$, the condition $u\equiv P$ implies
$u_\theta=0$, hence $H_1$ is real-valued.
Since it is holomorphic in $D^2$, it is also real-valued in $D^2$, hence constant. From
$H_1(0)=0$, we obtain $H_1\equiv0$, consequently $H\equiv0$ and conformality follows. Finally, conformality implies $u_r=0$ on $\partial D^2$. Then, inductively, $\nabla^k u=0$ on $\partial D^2$ for every $k$, and $u\equiv P$ by unique continuation for harmonic maps, see \cite{Sam}.

This argument fails on any annulus $A=\{\rho<|z|<1\}$ ($\rho\in (0,1)$), since the condition $H_1(0)=0$ is essential to rule out the possibility $H_1(z)=z^2H(z)\equiv const\ne 0$.

This obstruction is realized by the maps given in \cite[Example 5.1]{lem}, as we now recall. Identify the annulus conformally with a flat cylinder
$ [-T,T]\times \sph^1,$ parametrised by $(t,\theta)$
and consider maps of the form
\[
        u(t,\theta)=
        \bigl(\sin F(t)\cos \theta,\ \sin F(t)\sin \theta,\ \cos F(t)\bigr)\in \sph^2\subset\mathbb{R}^3.
\]
These are harmonic if and only if they satisfy the equation
\[
        F''=\sin F\cos F ,
\]
which can be integrated to
\[
        (F')^2-\sin^2 F=\lambda^2
\]
for some constant $\lambda$. For any $T>0$, $j\in \mathbb{N}\setminus\{0\}$ one may choose
$\lambda_j>0$ so that the solution satisfies
\begin{equation}\label{condF}
F(-T)=0, \qquad F(T)=2\pi j .
\end{equation}
Then the corresponding function $u_j$ satisfies $u_j\equiv (0,0,1)\in \sph^2$ on the boundary, yet it is non-constant. In this case one can compute $H_1(z)=\lambda^2_j$ and the map is not conformal.

Let us also note that by rotational symmetry and the boundary condition \eqref{condF}, we have $\deg(u_j)=0$ for every $j$, where we see $u_j$ as a map from $\R^2$ into $\sph^2$ constant outside the annulus $A$, and its Brouwer degree is defined lifting $u_j$ to a map from $\sph^2$ into $\sph^2$ via stereographic projection. 
Moreover every $u_j$ is homotopically equivalent to the constant map, relative to their constant boundary condition (the homotopy breaks the rotational symmetry, though).

\medskip

The proof of conformality in Theorem \ref{thmx2} is different and relies on $L^2$-integrability of $\nabla\tilde u$ and on regularity of $u$, hence of $\tilde u$ and of the Hopf differential in $\overline{\mathbb{C}_+}\setminus\{\pm1\}$ rather than on topology. Indeed, one defines ${\tilde H}$ similar to \eqref{defH}:
$$\tilde H(z)=\left(|\partial_x \tilde u|^2-|\partial _y \tilde u|^2\right)-2i\left(\partial _x \tilde u\cdot \partial_y \tilde u\right).$$
This is holomorphic and integrable on $\overline{\mathbb{C}_+}\setminus\{\pm 1\}$, real-valued on $\partial \mathbb{C}_+\setminus\{\pm1\}$ and therefore can be Schwarz-reflected ($\tilde H(\overline{z})=\overline{\tilde H(z)}$) to a holomorphic map,  still denoted by $\tilde H$, on $\mathbb{C}\setminus\{\pm1\}$ (using the regularity of half-harmonic maps inside $I$, but \emph{not} on $\partial I$, see \cite{DLR2}).
The integrability of $\tilde H$ near $\pm1$ implies that it has at most simple poles at $\pm 1$, so it can be written as $\tilde H(z)=\frac{\tilde H_1(z)}{(z-1)(z+1)}$ with $\tilde H_1$ holomorphic on $\mathbb{C}$. The integrability of $\tilde H$ at infinity and Cauchy's formula then imply that $\tilde H_1'\equiv 0$ and eventually $\tilde H_1\equiv 0$, so that $\tilde u$ is conformal, hence, as before, constant by unique continuation.

This proof breaks if we replace $I=(-1,1)$ with $I\cup J=(a,b)\cup (c,d)$, where $a<b<c<d$, since in this case the integrability of $\tilde H(z)=\frac{\tilde H_1(z)}{(z-a)(z-b)(z-c)(z-d)}$ does not imply the vanishing of $\tilde H_1$ (take e.g. $\tilde H_1=const\ne 0$). This raises the question whether on the union of two disjoint intervals there exists non-constant half-harmonic maps which are constant outside. This is exactly what we are going to prove.

\begin{thm}\label{exnonconst} For every {$k\in \mathbb{N}\setminus\{0\}$}, there exists $a_0\in (0,1)$, such that for every $a\in (0,a_0]$ there are $k$ distinct non-constant maps $u_1,\dots, u_k\in \dot H^{\frac12}(\R,\sph^1)$ which are half-harmonic in $I_-\cup I_+:=(-1,-a)\cup (a,1)$  and equal to $1 \in \sph^1\subset\mathbb{C}$ in $\R\setminus (I_-\cup I_+)$.
 \end{thm}

The difficulty is that, contrary to Lemaire's example, we cannot reduce the half-harmonic map equation to an ODE and we do not have explicit examples. Here, instead, we will  use a minimization procedure, but in \emph{local} homotopy classes, i.e. prescribing $\deg(u_j, I_\pm)= \pm j$. Then, similar to Lemaire's functions, we have $\deg(u_j)=0$ as a global degree of maps from $\R$ into $\sph^1$ constant outside  $I_\pm$, for $1\le j\le k$, but no $u_j$ can be deformed to a different $u_{j'}$ or to a constant, relative to the constant boundary condition.

\medskip

Compared to Lemaire's example, where for each annulus $A_{\rho,1}=\{\rho<|z|<1\}$, $\rho\in (0,1)$, there are infinitely many distinct harmonic maps $u:A_{\rho,1}\to \sph^2$ with $u|_{\partial A_{\rho,1}}\equiv (0,0,1)$ we have the following natural open questions:

\begin{pb} {Given $k\in \mathbb{N}\setminus \{0\}$}, does Theorem \ref{exnonconst} hold also for \emph{any} $a_0\in (0,1)$, or only for $a_0$ sufficiently small?
\end{pb} 

\begin{pb} In Theorem \ref{exnonconst}, for a fixed $a\in (0,1)$ are there infinitely many half-harmonic maps or just finitely many?
\end{pb}

\medskip

The methods of Theorem \ref{exnonconst} can be used to prove also a multiplicity result for non-constant boundary data with energy below the critical threshold of $2\pi$.

 \begin{thm}\label{exk} Given {$k\in \mathbb{N}\setminus \{0\}$}  and $\Lambda\in [0,2\pi)$ there exists $a_0>0$, such that for every $a\in (0,a_0]$ and $g\in \dot H^{\frac12}(\R,\sph^1)$ with $E(g)\le\Lambda$  there exists $k+1$ distinct maps $u_0,\dots, u_k\in \dot H^{\frac12}(\R,\sph^1)$ which are half-harmonic in $I_-\cup I_+=(-1,-a)\cup (a,1)$ and equal to $g$ in $\R\setminus (I_-\cup I_+)$. Moreover $E(u_j)<2\pi$ and in particular $\deg(u_j)=0$ for $j=0,\dots,k$.
 \end{thm}


Again Theorem \ref{exk} will be proven by minimizing with prescribed local degree. This will be defined as a degree relative to $g$, using the multiplicative structure of $\sph^1\subset\mathbb{C}$.

\subsection{Multiplicity results on a single interval}

Let us now consider half-harmonic maps on $I=(-1,1)$, as in \cite{HM}.
The problem of which homotopy classes $\M{E}_{g,d}$ contain a minimizer of $E$ is still largely open. In the case of harmonic maps from the disk into $\sph^2$, while Brezis and Coron \cite{bc} proved that there exists a boundary datum $g:\partial D^2\to \sph^2$ (any parallel of $\sph^2$ parametrised with constant speed) for which the Dirichlet energy can be minimised in only $2$ different homotopy classes,  Kuwert gave a complete characterisation of which homotopy classes admit a minimiser.

\begin{thmx}[Kuwert \cite{kuw}]\label{thmkuw} Consider $g\in C^\infty(\partial D^2,\sph^2)$ and
$$\tilde{\mathcal{E}}_{g,d}:=\{u\in H^{1}(D^2,\sph^2):u|_{\partial D^2}=g,\deg(u,u_0)=d\},$$
where $u_0$ is a fixed extension of $g$ {and $\deg(u,u_0)$ denotes the degree of $u$ relative to $u_0$, suitably defined}. If $g$ has a conformal extension $u^+$ to $D^2$, call $d^+:=\deg(u^+,u_0)$, otherwise set $d^+:=+\infty$. If $g$ has an anticonformal extension $u^-$ to $D^2$, call $d^-:=\deg(u^-,u_0)$, otherwise set $d^-:=-\infty$. Then $E$ has a minimizer in $\tilde{\mathcal{E}}_{g,d}$ if and only if $d^-\le d\le d^+$.
\end{thmx}

Interestingly, there are non-constant boundary data $g:\partial D^2\to \sph^2$ which have both a conformal and anticonformal extension, e.g. $g(e^{i\theta})=P^{-1}(e^{i d\theta})$, where $P:\sph^2 \to \mathbb{C}$ is the stereographic projection from the north pole, so that for every $k\in \mathbb{N}$ there exists a boundary datum $g$ giving rise to minimizers in exactly $k$ homotopy classes, but generically there are minimizers in infinitely many homotopy classes.

\medskip

For the fractional Dirichlet energy in dimension $1$, the first result about the existence of minimizers in more than 2 homotopy classes that we have is the following straightforward consequence of our methods in \cite{HM}. First let us recall that half-harmonic maps in $I$ which are $C^0$ in $\R\setminus I$ are continuous in $\bar I$, see e.g. \cite{Sch} and Lemma \ref{bdryreg}.
 
\begin{thm}\label{exist3} Consider $g\in \dot H^{\frac12}(\R,\sph^1)\cap C^0(\R,\sph^1)$ non-constant outside $I=(-1,1)$, such that $g(-1)=g(1)$. Let $u_0$ minimize $E$ in $\M{E}_{g}$ and assume that $\deg(u_0,\bar I)=0$ (well defined thanks to regularity and $u_0(-1)=u_0(1)$). Call $d_0:=\deg(u_0)$. Then there are at least $3$ half-harmonic maps {in $\M{E}_g$}: the absolute minimizer $u_0\in \M{E}_{g,d_0}$, and maps $u_{\pm}$ minimizing in $\mathcal{E}_{g,d_0\pm 1}$.
\end{thm}

The assumption $\deg(u_0,\bar I)=0$ of Theorem \ref{exist3} is satisfied, for instance, if $\deg(g,\R\setminus I):=\deg(g)-\deg(g,\bar I)=0$ and $E(g)\le 2\pi$, because in this case any absolute minimizer $u_0$ must have energy at most $2\pi$. Therefore,   either $\deg(u_0)=0$ or  $\deg(u_0)=\pm1$.  The latter case is impossible since $E(u_0)=2\pi$  would imply that $u_0$ is a  Blaschke product of degree $\pm1$, contradicting the hypothesis $g(1)=g(-1)$. But then $\deg(u_0,\bar I)=\deg(u_0)-\deg(g,\R\setminus I)=0$.

\medskip

We now focus on special boundary data, namely the restrictions of Blaschke products, namely holomorphic maps $\tilde h:\cl{\mathbb{C}_+}\to \cl{D^2} \subset \mathbb{C}$ of the form
\begin{equation}\label{Blas}
\tilde h(z)= e^{i\theta_0}\prod_{j=1}^k\frac{z-\beta_j}{z-\bar \beta_j},\quad \beta_1,\dots\beta_k\in \mathbb{C}_+,\, \theta_0\in \R.
\end{equation}
Call $h=\tilde h|_{\R\times\{0\}}$. Then
$$E(h)=\int_{\mathbb{C}_+} |\nabla \tilde h|^2 |dz|^2 =2\pi \deg (h)=2\pi k.$$
We define for $k>0$ the set of Blaschke products
$$\M{B}_k=\left\{\tilde h\text{ as in }\eqref{Blas}\right\} $$
and the set of their conjugates
$$\M{B}_{-k}=\left\{\bar{\tilde h}: \tilde h \in \M{B}_k\right\}.$$
For $k=0$, $\M{B}_0$ are just constant maps into $\sph^1$.

\medskip

Notice that for $\tilde h\in \M{B}_{\pm k}$ and $h(x)=\tilde h(x,0)$, we have $\deg(h)=\pm k$, $E(h)=2\pi|k|$.
The following non-existence result is a generalization of \cite[Thm. 1.3]{HM}.

\begin{thm}\label{thmBl} Let $g=\tilde h|_{\R\times\{0\}}$ for some $\tilde h\in \mathcal{B}_k$, $k\ne 0$. Then
\begin{equation}\label{enBla}
\inf_{\mathcal{E}_{g,k}} E= 2\pi |k| = E(g).
\end{equation}
Moreover, for $d>k>0$, or for $d<k<0$,
\begin{equation}\label{enBla2}
\inf_{\mathcal{E}_{g,d}} E= 2\pi |d|,
\end{equation}
and it is not attained.
\end{thm}

Importantly, any $g$ for which \eqref{enBla} is satisfied is the trace of a Blaschke product (see e.g. \cite[Lemma 3.1]{BMRS} or \cite[Lemma 5.1]{HM}), either holomorphic or anti-holomorphic, but not both (except for the case $k=0$, $g\equiv const$). Then, contrary to the case of harmonic maps from $D^2$ into $\sph^2$, in the half-harmonic case we do not have any example of boundary data $g$ for which there are minimizers in $\mathcal{E}_{g,k}$ for finitely many $k$'s only.

\begin{pb} Is there any $g\in \dot H^\frac12(\R,\sph^1)$ such that the infimum of $E$ in $\mathcal{E}_{g,k}$ 
is attained for finitely many values of $k\in \mathbb{Z}$?
\end{pb}

Notice that \eqref{enBla} and $\inf_{\mathcal{E}_{g,0}} E>0$ suggest that the intermediate energy levels might be separated by jumps smaller than $2\pi$, hence, by the same compactness argument as in \cite{HM}, there would be minimizers in $k+1$ homotopy classes.

\begin{pb} Given $g=\tilde h|_{\R\times\{0\}}$ for some $\tilde h\in \mathcal{B}_k$, $k> 0$ (and similarly for $k<0$) and setting
$$m_{g,j}:= \inf_{\mathcal{E}_{g,j}} E,$$
is it true that
\begin{equation}\label{mj}
|m_{g,j+1}-m_{g,j}|<2\pi\quad \text{for }0\le j\le k-1,
\end{equation}
and, therefore, $E$ has a minimizer in $\mathcal{E}_{g,j}$ for $0\le j\le k$?
\end{pb}

This is open even in the case $\tilde h_\ve(z):=\left(\frac{z-i\ve}{z+i\ve}\right)^k$, $g=g_\ve:=\tilde h_\ve |_{\R\times\{0\}}$ and $\ve>0$ is take arbitrarily small. In this case, since $g_\ve\to 1$ uniformly in $\R\setminus I$ as $\ve\to 0$, we easily see that $m_{g_\ve,0}\to 0^+$ as $\ve\to 0$ and we might expect
$$0<m_{g_\ve,j+1}-m_{g_\ve,j}<2\pi\quad \text{for }0\le j\le k-1,$$
for any $\ve>0$ sufficiently small. In this case $m_{g_\ve,j}$ would be attained for $0\le j\le k$.

While this is open in general, the following result gives Blaschke data for which $m_{g,j}$ is attained for up to four values of $j\in \{0,\dots,k\}$.

\begin{thm}\label{Bl3}  Consider $\tilde h\in\M B_k$ and $g=\tilde h|_{\R\times\{0\}}$,
and set
\[
q:=\frac1{2\pi}\int_{\R\setminus I} g\wedge g'\,dx\in (0,k),
\]
where for two vectors $v,w\in \R^2\simeq \mathbb{C}$ we define $v\wedge w:=v_1w_2-v_2w_1\in \R$.
We have:
\begin{itemize}
\item if $q\in (k-1,k)$, then $m_{g,d}$ is achieved for $d=k-1,k$;
\item if $q\in (k-2,k-1]$, then $m_{g,d}$ is achieved for $d=k-2,k-1,k$;
\item if $q=\ell \in\{1,\dots, k-2\}$, then $m_{g,d}$ is achieved for $d=\ell-1,\ell,\ell+1,k$;
\item if $q\in (\ell,\ell+1)$ with $\ell\in \{0,\dots, k-2\}$, then $m_{g,d}$ is achieved for $d=\ell,\ell+1,k$.
\end{itemize}
\end{thm}

We found no technique which allows to prove the existence of minimisers in more than $4$ degree classes in Theorem \ref{Bl3}. Moreover, except for $d>k$, in all cases in which we cannot prove that $m_{g,d}$ is achieved (for instance, when $d<0$), we cannot prove that it is not achieved either.




\section{Local and relative degree, their additivity and jumps}

\subsection{Product structure and degree additivity}

We will need a few lemmas.

\begin{lem}\label{product}
If $v,w\in\dot H^{\frac12}(\R,\sph^1)$, then $vw\in\dot H^{\frac12}(\R,\sph^1)$ and
\begin{equation}\label{Evw}
\sqrt{E(vw)}\le \sqrt{E(v)}+\sqrt{E(w)}.
\end{equation}
\end{lem}

\begin{proof}
Since $|v|=|w|=1$, for a.e. $x,y\in\R$ we have
\[
v(x)w(x)-v(y)w(y)=w(x)(v(x)-v(y))+v(y)(w(x)-w(y)),
\]
hence
\begin{equation}\label{vw}
|v(x)w(x)-v(y)w(y)|
\le |v(x)-v(y)|+|w(x)-w(y)|.
\end{equation}
Dividing \eqref{vw} by $|x-y|$, taking the $L^2(\R\times\R)$ norm and using Minkowski's inequality gives
\[\begin{split}
\left(\iint_{\R\times\R}\frac{|v(x)w(x)-v(y)w(y)|^2}{|x-y|^2}\,dxdy\right)^{1/2}
\le&
\left(\iint_{\R\times\R}\frac{|v(x)-v(y)|^2}{|x-y|^2}\,dxdy\right)^{1/2}\\
&
+
\left(\iint_{\R\times\R}\frac{|w(x)-w(y)|^2}{|x-y|^2}\,dxdy\right)^{1/2},
\end{split}
\]
and \eqref{Evw} follows from \eqref{12energybis}.
\end{proof}

\begin{lem}\label{lemunvn}
Let $u_n\to u$ and $v_n\to v$ in $H^{\frac12}(\sph^1,\sph^1)$.
Then
$ u_nv_n\to uv$ in $H^{\frac12}(\sph^1,\sph^1)$
\end{lem}

\begin{proof}
Let $a_n\to 0$  in
$H^{\frac12}(\sph^1,\mathbb C)$ with
\[
\sup_n\|a_n\|_{L^\infty(\sph^1)}<\infty,
\]
and let $w\in H^{\frac12}(\sph^1,\mathbb C)\cap L^\infty(\sph^1,\mathbb C)$.
We claim that
\begin{equation}\label{auxclaim}
a_nw\to0
\quad\text{in }H^{\frac12}(\sph^1,\mathbb C).
\end{equation}
Indeed,
\[
\|a_nw\|_{L^2(\sph^1)}
\le
\|w\|_{L^\infty(\sph^1)}\|a_n\|_{L^2(\sph^1)}
\to0.
\]
Moreover,
\[
|a_n(x)w(x)-a_n(y)w(y)|
\le
|w(x)|\,|a_n(x)-a_n(y)| +
|a_n(y)|\,|w(x)-w(y)|.
\]
Therefore,
\[
[a_nw]_{H^{\frac12}(\sph^1)}^2
\le
2\|w\|_{L^\infty(\sph^1)}^2[a_n]_{H^{\frac12}(\sph^1)}^2
+
2\int_{\sph^1}|a_n(y)|^2K_w(y)\,d\sigma(y),
\]
where
\[
K_w(y)
:=
\int_{\sph^1}
\frac{|w(x)-w(y)|^2}{|x-y|^2}\,d\sigma(x).
\]
Since $w\in H^{\frac12}(\sph^1)$, we have $K_w\in L^1(\sph^1)$. Setting
\[
M_0:=\sup_n\|a_n\|_{L^\infty(\sph^1)},
\]
for every $A>0$ we have
\[
\begin{aligned}
\int_{\sph^1}|a_n|^2K_w\,d\sigma
&=
\int_{\{K_w\le A\}}|a_n|^2K_w\,d\sigma
+
\int_{\{K_w>A\}}|a_n|^2K_w\,d\sigma \\
&\le
A\|a_n\|_{L^2(\sph^1)}^2
+
M_0^2\int_{\{K_w>A\}}K_w\,d\sigma.
\end{aligned}
\]
First choosing $A$ sufficiently large, and then letting $n\to\infty$ yields
\[
\int_{\sph^1}|a_n|^2K_w\,d\sigma\to0.
\]
This proves \eqref{auxclaim}.

Now set
\[
a_n:=u_n-u,
\qquad
b_n:=v_n-v,
\]
and write
\begin{equation}\label{unvn}
u_nv_n-uv
=
a_nv+ub_n+a_nb_n.
\end{equation}
Since $u_n,u,v_n,v$ are $\sph^1$-valued, we have
$\|a_n\|_{L^\infty(\sph^1)}
\le2$,
and $\|b_n\|_{L^\infty(\sph^1)}
\le2$, and by \eqref{auxclaim},
\[
a_nv\to0,
\qquad
ub_n\to0
\quad\text{in }H^{\frac12}(\sph^1,\mathbb{C}).
\]
It remains to estimate $a_nb_n$ on the right-hand side of \eqref{unvn}. We have
\[
\|a_nb_n\|_{L^2(\sph^1)}
\le
\|a_n\|_{L^\infty(\sph^1)}\|b_n\|_{L^2(\sph^1)}
\to0.
\]
Furthermore,
\[
\begin{aligned}
|a_n(x)b_n(x)-a_n(y)b_n(y)|
&\le
|a_n(x)|\,|b_n(x)-b_n(y)| \\
&\quad+
|b_n(y)|\,|a_n(x)-a_n(y)|,
\end{aligned}
\]
and hence
\[
[a_nb_n]_{H^{\frac12}(\sph^1)}^2
\le
2\|a_n\|_{L^\infty(\sph^1)}^2[b_n]_{H^{\frac12}(\sph^1)}^2
+
2\|b_n\|_{L^\infty(\sph^1)}^2[a_n]_{H^{\frac12}(\sph^1)}^2
\to0,
\]
which yields $a_nb_n\to0$ in $H^{\frac12}(\sph^1)$. This completes the proof.
\end{proof}

\begin{lem}\label{lem:degree-additivity-product}
For every $u,v\in \dot H^{\frac12}(\mathbb R,\sph^1)$ one has
\begin{equation}\label{degvw}
        \deg(uv)=\deg(u)+\deg(v),
\end{equation}
or, equivalently,
\[
        \deg(u)=\deg(v)+\deg(u\bar v).
\]
\end{lem}

\begin{proof} We give two independent proofs of this lemma.
Up to stereographic projection, we can consider $u,v$ as maps in $H^\frac12(\sph^1,\sph^1)$. 

If $u,v\in C^1(\sph^1,\sph^1)$, \eqref{degvw} follows from
$$\overline{uv}(uv)'=\bar u u' + \bar v v'$$
(here $u'=du/d\theta$, etc...) and the classical integral formula (see e.g. \cite[Sec. 2]{HM})
\[
        \deg(uv)        =     \frac{1}{2\pi i}\int_{\sph^1}\overline{uv}(uv)'\,d\theta =        \frac{1}{2\pi i}\int_{\sph^1}\bar u u'\,d\theta        +        \frac{1}{2\pi i}\int_{\sph^1}\bar v v'\,dx=\deg(u)+\deg(v) .\]
For the general case, consider sequences $(u_n),(v_n)\subset C^1(\sph^1,\sph^1)$ with $u_n\to u$ and $v_n\to v$ in $H^\frac12(\sph^1,\sph^1)$ (these exist by e.g. \cite[Thm. 1.5]{AM} or \cite[Prop. A.2]{HM}). By continuity of the degree in $H^\frac12(\sph^1,\sph^1)$, and since $u_n v_n\to uv$ in $H^\frac12(\sph^1,\sph^1)$ by Lemma \ref{lemunvn}, we have
$$\deg(uv)=\lim_{n\to +\infty} \deg(u_nv_n)=\lim_{n\to \infty} (\deg(u_n)+\deg(v_n))=\deg(u)+\deg(v).$$

An alternative proof, uses the lifting theorem of Brezis--Nirenberg
(Theorem 3 and Remark 10(iii) in \cite{BN}), according to which, every
$u\in \operatorname{VMO}(\sph^1,\sph^1)$ admits a representation
\[
        u(e^{i\theta})=e^{ik\theta+i\varphi(\theta)},
        \qquad k=\deg u,\quad \varphi\in \operatorname{VMO}(\sph^1,\mathbb R).
\]
Since $H^{\frac12}(\sph^1)\subset \operatorname{VMO}(\sph^1)$ and the
$H^{\frac12}$ degree agrees with the VMO degree, see
\cite[Proposition 2.1]{HM}, this applies to
$H^{\frac12}$-maps. Hence, writing
\[
        u=e^{ia\theta+i\varphi},\qquad
        v=e^{ib\theta+i\psi},
\]
with $\varphi,\psi\in \operatorname{VMO}(\sph^1)$, $a=\deg(u)$, $b=\deg(v)$,  we obtain
\[
        uv=e^{i(a+b)\theta+i(\varphi+\psi)},
\]
and therefore
\[
        \deg(uv)=a+b=\deg(u)+\deg(v).
\]
\end{proof}

\subsection{Local and relative degrees}

Given $a\in (0,1)$, set $I_+=I_+(a):=(a,1)$, $I_-=I_-(a):=(-1,-a)$. Unless necessary, we will drop the $a$ and simply write $I_+$ and $I_-$. For $u\in \dot H^{\frac12}(\R,\sph^1)$ such that $u\equiv 1$ on $\R\setminus (I_-\cup I_+)$ define $\deg(u,I_\pm)$ as follows. First set
\[
P_{\pm}u:=
\begin{cases}
u(x), & x\in I_\pm,\\
1, & x\notin I_\pm.
\end{cases}
\]
Then, by Lemma \ref{Tpm} below applied to $f=u-1$, $P_{\pm}u\in \dot H^{\frac12}(\R,\sph^1)$, hence we can define its degree.
\begin{defn}\label{defdeg} Given $u\in \dot H^{\frac12}(\R,\sph^1)$ such that $u\equiv 1$ on $\R\setminus (I_-\cup I_+)$ we define
$$\deg(u,I_\pm):= \deg(P_{\pm} u),$$
where, being $P_{\pm} u\in \dot H^{\frac12}(\R,\sph^1)$,  $\deg(P_{\pm} u)$ is defined as in \eqref{defdeg1}.
\end{defn}

As one would expect, the sum of local degrees gives the global degree:

\begin{lem}\label{add} For $u\in \dot H^\frac12(\R,\sph^1)$ with $u\equiv 1$ outside $I_-\cup I_+$, we have 
$$\deg(u) =\deg(u,I_+)+\deg(u,I_-).$$
\end{lem}

\begin{proof} We can write $u=P_+u P_-u$, and the lemma follows from Lemma \ref{lem:degree-additivity-product}. 
\end{proof}

Now, given $k_+,k_-\in \mathbb{Z}$, we set
$$\mathcal{E}_{k_+,k_-,1}=\mathcal{E}_{k_+,k_-,1}{(a)}:=\{u\in \dot H^{\frac12}(\R,\sph^1): u\equiv 1 \text{ in }\R\setminus (I_-(a)\cup I_+(a)),\, \deg(u, I_\pm(a))=k_{\pm}\}.$$

Using the product structure on $\sph^1$
we can also give the following relative definition of local degree. Consider $g,u \in \dot H^{\frac12}(\R,\sph^1)$ and assume that $u=g$ a.e. in $\R\setminus (I_-\cup I_+)$. Then, by Lemma \ref{product}, $v:=u\bar g\in\dot H^{\frac12}(\R,\sph^1)$.  
 Moreover $v\equiv 1$ in $\R\setminus (I_-\cup I_+)$. { Then, following Definition \ref{defdeg}:
\begin{defn} Given $g,u \in \dot H^{\frac12}(\R,\sph^1)$ with $u=g$ a.e. in $\R\setminus (I_-\cup I_+)$, we set
\[
\deg_g(u,I_\pm):=\deg(u\bar g,I_\pm)=\deg(P_{\pm} (u\bar g)),
\]
and for $k_\pm\in\mathbb Z$, we set
\[
\mathcal E_{k_+,k_-,g}:=
\left\{u\in \dot H^{\frac12}(\R,\sph^1):u=g\text{ in }\R\setminus (I_-\cup I_+),
\ \deg_g(u,I_\pm)=k_\pm\right\}
=g\,\mathcal E_{k_+,k_-,1}.
\]
\end{defn}}


\subsection{Degree jumps}

We now want to estimate the cost of degree jumps. Let us start recalling from \cite{HM}:
\begin{prop}[Prop. 3.1 in \cite{HM}]\label{p:jump} Let $(u_n)\subset \mathcal{E}_k:=\{v\in \dot H^{\frac12}(\R,\sph^1):\deg(v)=k\}$, converge weakly in $\dot H^{\frac{1}{2},2}(\R,\sph^1)$ to $u\in \mathcal{E}_{\ell}$. Then $$E(u)\le \liminf_{n\to\infty} E(u_n)-2\pi |k-\ell|.$$
\end{prop}
In Proposition \ref{p:jump} the degree jump is global. We now consider a sequence of maps  $(u_n)\subset \mathcal{E}_{k_+,k_-,1}$ whose degrees jump to $(\ell_+,\ell_-)$ independently on $I_+$ and $I_-$. We show that, even if the total degree might not jump ($k_+ + k_- = \ell_+ +\ell_-$), so that Proposition \ref{p:jump} does not imply any cost, the jump has indeed a cost of $2\pi |k_+-\ell_+|+2\pi |k_--\ell_-|$, as we are now going to show.

\begin{prop}[Local degree jumps]\label{ldegjump}
Let $g\in \dot H^{\frac12}(\mathbb R,\sph^1)$. Assume that
$ u_n \rightharpoonup u$
weakly in  $\dot H^{\frac12}(\mathbb R,\sph^1)$,
with $u_n\in \mathcal E_{k_+,k_-,g}$,
   $u\in \mathcal E_{\ell_+,\ell_-,g}$.
Then
\[
        E(u)
        \le
        \liminf_{n\to\infty} E(u_n)
        -
        2\pi\bigl(|k_+-\ell_+|+|k_- - \ell_-|\bigr).
\]
\end{prop}

Since the proof of Proposition \ref{p:jump} is based on a nonlocal definition of degree, its proof does not extend easily to Proposition \ref{ldegjump}. We will need to introduce suitable projection operators and prove their continuity.

\medskip

Consider the closed linear subspace $X$ of $H^{\frac12}(\mathbb R,\mathbb{C})$ defined as
\begin{equation}\label{defX}
X:=\bigl\{f\in H^{\frac12}(\mathbb R,\mathbb{C}): \supp f \subset {\overline{I_+\cup I_-}}\bigr\}.
\end{equation}
By the fractional Poincar\'e inequality, the seminorm $[\,\cdot\,]_{\dot H^{\frac12}}$
is a norm on $X$, equivalent to the full $H^{\frac12}$-norm. In particular,
$X$ is a Hilbert space with the scalar product

\begin{equation}\label{scalarprod}
\langle f_1,f_2\rangle_{\dot H^{\frac12}}
:=
\frac{1}{2\pi}
\int_{\mathbb R}\int_{\mathbb R}
\frac{(f_1(x)-f_1(y))\cdot(f_2(x)-f_2(y))}
{|x-y|^2}\,dx\,dy.
\end{equation}

\begin{lem}\label{Tpm}
Let $X$ be as in \eqref{defX}. Then the operators $T_\pm : X\to X$ given by
\[
T_\pm f(x):=
\begin{cases}
f(x), & x\in I_\pm\\
0, & x\notin I_\pm
\end{cases}
\]
are bounded.
\end{lem}

\begin{proof}
For $f\in X$ one has
\[
[T_+f]_{\dot H^{\frac12}}^2
\le [f]_{\dot H^{\frac12}}^2
   +4\iint_{I_+\times I_-}\frac{|f(x)||f(y)|}{|x-y|^2}\,dx\,dy.
\]
Since $\mathrm{dist}(I_+,I_-)=2a:=d>0$, we have $|x-y|\ge d$ on $I_+\times I_-$, hence
\[
\iint_{I_+\times I_-}\frac{|f(x)||f(y)|}{|x-y|^2}\,dx\,dy
\le \frac{1}{d^2}\|f\|_{L^1(I_+)}\|f\|_{L^1(I_-)}
\le C\|f\|_{L^2(I_+\cup I_-)}^2.
\]
Since $I_+\cup I_-$ is bounded and $f$ vanishes outside it, the fractional Poincar\'e inequality yields
\[
\|f\|_{L^2(I_+\cup I_-)}^2 \le C [f]_{\dot H^{\frac12}}^2.
\]
Therefore
\[
[T_+f]_{\dot H^{\frac12}} \le C [f]_{\dot H^{\frac12}}.
\]
The same argument applies to $T_-$.
\end{proof}

\medskip

\noindent\emph{Proof of Proposition \ref{ldegjump}.}
Set
\[
        h_n:=u_n-g,
        \qquad
        h:=u-g.
\]
Let $X$ be as in \eqref{defX} and $T_\pm:X\to X$ be as in Lemma \ref{Tpm}. Since $u_n=u=g$ in $\mathbb R\setminus (I_-\cup I_+)$, we have $h_n,h\in X$. 
Define
\[
        h_{n,\pm}:=T_\pm h_n,
        \qquad
        h_\pm:=T_\pm h,
\]
and
\[
        u_{n,\pm}:=g+h_{n,\pm},
        \qquad
        u_\pm:=g+h_\pm.
\]
Since $T_\pm$ is bounded and $h_n \rightharpoonup h$ weakly in $X$, we get $h_{n,\pm} \rightharpoonup h_{\pm}$ weakly in $X$, hence
\[
        u_{n,\pm}\rightharpoonup u_\pm
        \qquad\text{weakly in } \dot H^{\frac12}(\mathbb R,\mathbb C).
\]
Also $u_{n,\pm},u_\pm$ are $\sph^1$-valued, because
$u_{n,\pm}=u_n$ in $I_\pm$ and $u_{n,\pm}=g$ in
$\mathbb R\setminus I_\pm$, and similarly for $u_\pm$.  

We claim that
\[
        \deg(u_{n,\pm})=\deg(g)+k_\pm,
        \qquad
        \deg(u_\pm)=\deg(g)+\ell_\pm.
\]
Indeed,
\[
        u_{n,\pm}\bar g
        =
        \begin{cases}
        u_n\bar g & \text{in } I_\pm,\\
        1 & \text{in } \mathbb R\setminus I_\pm,
        \end{cases}
        =
        P_\pm(u_n\bar g),
\]
and therefore
\[
        \deg(u_{n,\pm}\bar g)=\deg_g(u_n,I_\pm)=k_\pm.
\]
By Lemma \ref{lem:degree-additivity-product},
\[
        \deg(u_{n,\pm})
        =
        \deg(g)+\deg(u_{n,\pm}\bar g)
        =
        \deg(g)+k_\pm.
\]
The proof for $u_\pm$ is identical.

Applying Proposition \ref{p:jump} to the sequences
$(u_{n,+})$ and $(u_{n,-})$, we get
\[
        E(u_+)
        \le
        \liminf_{n\to\infty}E(u_{n,+})
        -
        2\pi |k_+-\ell_+|,
\]
and
\[
        E(u_-)
        \le
        \liminf_{n\to\infty}E(u_{n,-})
        -
        2\pi |k_- -\ell_-|.
\]
Hence
\[
        E(u_+)+E(u_-)
        \le
        \liminf_{n\to\infty}
        \bigl(E(u_{n,+})+E(u_{n,-})\bigr)
        -
        2\pi\bigl(|k_+-\ell_+|+|k_- - \ell_-|\bigr).
\]

It remains to compare the localized energies with the full energy. We first extend $E$ to $\dot H^{\frac12}(\mathbb R,\mathbb C)$ by the same
quadratic expression and for $f_1,f_2\in \dot H^{\frac12}(\mathbb R,\mathbb C)$ we consider $\langle f_1,f_2\rangle_{\dot H^{\frac12}}$ defined as in \eqref{scalarprod}.
In particular,
\[
E(f)=\langle f,f\rangle_{\dot H^{\frac12}}
=\frac{1}{2\pi}[f]^2_{\dot H^{\frac12}}.
\]
Since
\[
u_n=g+h_{n,+}+h_{n,-},
\qquad
u_{n,\pm}=g+h_{n,\pm},
\]
bilinearity gives
\[
E(u_n)
=
E(u_{n,+})+E(u_{n,-})-E(g)
+
2\langle h_{n,+},h_{n,-}\rangle_{\dot H^{\frac12}}.
\]
The same identity holds for $u$, $u_+$ and $u_-$:
\[   E(u)  = E(u_+)+E(u_-)-E(g) + 2\langle h_+,h_-\rangle_{\dot H^{\frac12}}.\]
Since the sequence
$(h_n)$ is bounded in $X$ and supported in $I_-\cup I_+$, by the equivalence of the $\dot H^{\frac12}$ and
$H^{\frac12}$ norms on $X$ and by the compact embedding
$
H^{\frac12}(I_-\cup I_+)\hookrightarrow L^2(I_-\cup I_+),
$
we have
$
h_n\to h$ in $L^2(I_-\cup I_+)$.
Thus
\[
        h_{n,\pm}\to h_\pm
        \qquad\text{in } L^2(I_\pm).
\]
For $\operatorname{supp} f_1\subset \overline{I_+}$ and
$\operatorname{supp} f_2\subset \overline{I_-}$,  
\[
        \langle f_1,f_2\rangle_{\dot H^{\frac12}}
        =
        -\frac{1}{\pi}
        \int_{I_+}\int_{I_-}
        \frac{f_1(x)\cdot f_2(y)}{|x-y|^2}\,dx\,dy .
\]
Since $\operatorname{dist}(I_+,I_-)>0$, the kernel $|x-y|^{-2}$ is bounded
on $I_+\times I_-$, and the bilinear form $\langle \cdot , \cdot \rangle_{\dot H^{\frac12}}$ is continuous with respect to the strong $L^2(I_+)\times L^2(I_-)$ convergence. Hence
\[
        \langle h_{n,+},h_{n,-}\rangle_{\dot H^{\frac12}}
        \longrightarrow
        \langle h_+,h_-\rangle_{\dot H^{\frac12}}.
\]
Combining the previous estimates, we obtain
\[
\begin{aligned}
        E(u)
        &=
        E(u_+)+E(u_-)-E(g)
        +
        2\langle h_+,h_-\rangle_{\dot H^{\frac12}}
        \\
        &\le
        \liminf_{n\to\infty}
        \bigl(E(u_{n,+})+E(u_{n,-})\bigr)
        -
        2\pi\bigl(|k_+-\ell_+|+|k_- - \ell_-|\bigr)
        \\
        &\qquad
        -E(g)
        +
        2\lim_{n\to\infty}
        \langle h_{n,+},h_{n,-}\rangle_{\dot H^{\frac12}}
        \\
        &=
        \liminf_{n\to\infty}
        \Bigl(
        E(u_{n,+})+E(u_{n,-})-E(g)
        +
        2\langle h_{n,+},h_{n,-}\rangle_{\dot H^{\frac12}}
        \Bigr)
        \\
        &\qquad
        -
        2\pi\bigl(|k_+-\ell_+|+|k_- - \ell_-|\bigr)
        \\
        &=
        \liminf_{n\to\infty}E(u_n)
        -
        2\pi\bigl(|k_+-\ell_+|+|k_- - \ell_-|\bigr).
\end{aligned}
\]
This proves the proposition. \hfill$\square$

\section{Proof of Theorems \ref{exnonconst} and \ref{exk}}

Theorem \ref{exnonconst} follows from Theorem \ref{exk} by taking $g\equiv 1$, in which case the minimizer in $\mathcal E_{0,0,1}$ is the constant map $1$, while the minimizers in $\mathcal E_{-j,j,1}$, $1\le j\le k$, are non-constant and distinct because their local degrees are distinct and different from zero.

We will then focus on proving Theorem \ref{exk}.

\begin{lem}\label{lem:loglog-transition}
For $j,m\in \mathbb N$  set
\[
\rho_m:=\exp(-e^{2\pi m}),\qquad
\delta_{m,j}:=\rho_{m+j}=\exp(-e^{2\pi(m+j)}),
\]
and define
\[
\varphi_{m,j}(x):=
\begin{cases}
2\pi j,& |x|\le \delta_{m,j},\\[2mm]
\log\log\dfrac1{|x|} - 2\pi m,& \delta_{m,j}<|x|<\rho_m,\\[2mm]
0,& |x|\ge \rho_m.
\end{cases}
\]
Then
\[
[e^{i\varphi_{m,j}}]_{\dot H^{\frac12}}\le [\varphi_{m,j}]_{\dot H^{\frac12}}\longrightarrow0
\qquad\text{as }m\to\infty,
\]
for every fixed $j$.
\end{lem}

\begin{proof}
The first inequality follows from
\[
|e^{i\varphi(x)}-e^{i\varphi(y)}|\le |\varphi(x)-\varphi(y)|.
\]
It remains to estimate $[\varphi_{m,j}]_{\dot H^{\frac12}}$. We write, for simplicity, $\rho=\rho_m$, $\delta=\delta_{m,j}$ and $\varphi=\varphi_{m,j}$, and decompose
\[
A_1:=\{|x|\le\delta\},\qquad
A_2:=\{\delta<|x|<\rho\},\qquad
A_3:=\{|x|\ge\rho\}.
\]
Then
\begin{align*}
[\varphi]_{\dot H^{\frac12}}^2
&=\iint_{A_2\times A_2}\frac{|\varphi(x)-\varphi(y)|^2}{|x-y|^2}\,dxdy
+2\iint_{A_1\times A_2}\frac{|\varphi(x)-\varphi(y)|^2}{|x-y|^2}\,dxdy\\
&\quad+2\iint_{A_1\times A_3}\frac{|\varphi(x)-\varphi(y)|^2}{|x-y|^2}\,dxdy
+2\iint_{A_2\times A_3}\frac{|\varphi(x)-\varphi(y)|^2}{|x-y|^2}\,dxdy\\
&=:J_1+J_2+J_3+J_4.
\end{align*}
Since $f(x)=\log\log(1/|x|)$ belongs to $H^{\frac12}((-1/2,1/2))$ (see \cite[Lemma A.9]{HM}), the absolute continuity of the integral gives $J_1\to0$ as $\rho\to0$.

For $J_2$ we scale $x=\delta X$, $y=\delta Y$. If $|X|\le1$ and $1\le |Y|\le \rho/\delta$, then
\[
\varphi(\delta X)-\varphi(\delta Y)
=\log\left(1+\frac{\log |Y|}{\log\delta}\right)^{-1},
\]
and hence, since $0\le \log |Y|\le \log(\rho/\delta)<-\log\delta$,
\[
|\varphi(\delta X)-\varphi(\delta Y)|
\le C_j\frac{\log |Y|}{|\log\delta|}.
\]
Therefore
\[
J_2\le
\frac{C_j}{(\log\delta)^2}
\int_{|X|\le1}\int_{1\le |Y|\le \rho/\delta}
\frac{(\log |Y|)^2}{|X-Y|^2}\,dYdX
\le \frac{C_j}{(\log\delta)^2}\to0.
\]
Here we used the elementary bound
\[
\sup_{R>1}\int_{|X|\le1}\int_{1\le |Y|\le R}
\frac{(\log |Y|)^2}{|X-Y|^2}\,dYdX<\infty,
\]
which follows by splitting the region into $||Y|-1|<1$ and its complement.

For $J_3$ we use $|\varphi(x)-\varphi(y)|=2\pi j$ on $A_1\times A_3$ and obtain
\[
J_3\le C_j\int_{|x|\le\delta}\int_{|y|\ge\rho}\frac{dy\,dx}{|x-y|^2}
\le C_j\frac{\delta}{\rho}\to0.
\]
Finally, for $J_4$ we scale $x=\rho X$, $y=\rho Y$. If $\delta/\rho<|X|<1$ and $|Y|\ge1$, then
\[
\varphi(\rho X)-\varphi(\rho Y)
=\log\left(1+\frac{\log |X|}{\log\rho}\right),
\]
and therefore
\[
|\varphi(\rho X)-\varphi(\rho Y)|
\le C_j\frac{|\log |X||}{|\log\rho|}.
\]
Thus
\[
J_4\le
\frac{C_j}{(\log\rho)^2}
\int_{|X|\le1}\int_{|Y|\ge1}
\frac{(\log |X|)^2}{|X-Y|^2}\,dYdX
\le \frac{C_j}{(\log\rho)^2}\to0.
\]
This proves the lemma.
\end{proof}

\medskip

\noindent\emph{Proof of Theorem \ref{exk} (completed).}
Fix $k\in\mathbb N$, $\Lambda\in [0,2\pi)$. By Lemma~\ref{lem:loglog-transition}, we may choose $m$ so large that, with
\[
 v_j(x):=e^{i\varphi_{m,j}(x)},\qquad 0\le j\le k,
\]
one has   $$E(v_j)=\frac{\([v_j]_{\dot H^\frac12}\)^2}{2\pi}<\(\sqrt{2\pi}-\sqrt{\Lambda}\)^2,\qquad 0\le j\le  k.$$
Set $a_0:=\delta_{m,k}$, as given in Lemma \ref{lem:loglog-transition}. Then, given $a\in (0,a_0]$, $v_j\in {\mathcal E_{-j,j,1}(a)}$ for every {$0\le j\le k$}.
Therefore,  given $g$ with $E(g)\leq\Lambda$,
Lemma~\ref{product} gives
\[
\sqrt{E(g v_j)}\le \sqrt{E(g)}+\sqrt{E( v_j)}<\sqrt{2\pi},
\]
so that
\begin{equation}\label{eq:upper-bounds-mj}
\inf_{\mathcal E_{-j,j,g}}E\le E(g v_j)<2\pi,
\qquad 0\le j\le k.
\end{equation}
{Fix now $j$ with $0\le j\le k$ and} define
\[
m_j:=\inf_{\mathcal E_{-j,j,g}}E.
\]
Let $(u_n)\subset \mathcal E_{-j,j,g}$ be a minimizing sequence.    Then up to a subsequence, $ u_n\rightharpoonup u^{(j)}\text{ in }\dot H^{\frac12}(\mathbb R,\sph^1)$ for some $u^{(j)}\in \dot H^{\frac12}(\mathbb R,\sph^1) $ with $u^{(j)}=g$ on $\R\setminus (I_-\cup I_+)$. 
Write
\[
(\ell_+,\ell_-)
:=\bigl(\deg_g(u^{(j)},I_+),\deg_g(u^{(j)},I_-)\bigr).
\]
Proposition~\ref{ldegjump} gives
\begin{equation}\label{eq:jump-general-j}
E(u^{(j)})
\le m_j-2\pi s_j,
\end{equation}
where
\[
s_j:=|\ell_+ +j|+|\ell_- -j|.
\]
By \eqref{eq:upper-bounds-mj}  we have $m_j<2\pi$. Since $E(u^{(j)})\ge0$, \eqref{eq:jump-general-j} implies $s_j=0$. Therefore
\[
(\ell_+,\ell_-)=(-j,j),
\]
so $u^{(j)}\in \mathcal E_{-j,j,g}$ and $E(u^{(j)})=m_j$. In particular $u^{(j)}$ minimizes $E$ in $\mathcal E_{-j,j,g}(a)$.

Since smooth compactly supported variations inside $(-1,-a)\cup(a,1)$ preserve both the boundary condition on $\mathbb R\setminus ((-1,-a)\cup(a,1))$ and the local degrees, $u^{(j)}$ is a critical point of $E$ with respect to arbitrary compactly supported variations in $(-1,-a)\cup(a,1)$. Hence $u^{(j)}$ is half-harmonic in $(-1,-a)\cup(a,1)$.

The maps $u^{(0)},\dots,u^{(k)}$ are pairwise distinct, because their local relative degrees
are all different. For $j\ge1$, the map $u^{(j)}$ is non-constant, again because its local relative degrees are non-zero. On the other hand, since $m_j=E(u^{(j)})<2\pi$, we have $\deg(u^{(j)})=0$.

This completes the proof of Theorem~\ref{exk}, {hence also of Theorem \ref{exnonconst} for arbitrary $k\ge 1$.} \hfill $\square$

\section{Proof of Theorems \ref{exist3}, \ref{thmBl} and \ref{Bl3}}

\subsection{Proof of Theorem \ref{exist3}.}
Let $u_0\in \mathcal{E}_g$ be the absolute minimizer, and set $\deg(u_0)=k_0$. By regularity theory (see \cite{DLR, Sch}, Lemma \ref{bdryreg}), $u_0|_{\bar I}\in C^0(\bar I)\cap C^\infty(I)$. Since $u_0$ is non-constant (see e.g. \cite[Lemma 4.3]{HM}), $u_0(-1)=u_0(1)$, and $\deg(u_0,I)=0$, there are points $x_+,x_-\in I$ such that $\mathrm{sgn}(u_0(x_\pm)\wedge u_0'(x_\pm) )=\pm 1$. By \cite[Lemma 4.1]{HM},
$$\inf_{\mathcal{E}_{g, k_0\pm 1}}E< \inf_{\mathcal{E}_g} E +2\pi,$$
hence, as in the proof of Theorem 1.1 of \cite{HM}, there are minimizers $u_{\pm}$ of $E$ restricted to $\mathcal{E}_{g, k_0\pm 1}$, and such minimizers are half-harmonic maps. \hfill$\square$

\subsection{Proof of Theorem \ref{thmBl}.}

\begin{lem} \label{l1} {Given $g\in \dot H^\frac{1}{2}(\R,\sph^1)$,} set
\[
m_{g,j}:=\inf_{\M{E}_{g,j}}E .
\]
Then, for every $j,\ell\in\mathbb Z$,
\[
m_{g,\ell}\leq m_{g,j}+2\pi|j-\ell|.
\]
\end{lem}

\begin{proof}
This follows from the construction in the proof of Part (iv) of Theorem 1.3 in \cite{HM}: given $\ve>0$ it is possible to attach $|j-\ell|$ bubbles to any
$u\in \M{E}_{g,j}$, obtaining a map
$v\in \M{E}_{g,\ell}$ such that
\[
E(v)\le E(u)+2\pi |j-\ell|+\ve .
\]
Letting first $u$ approach $m_{g,j}$ and then $\ve\downarrow0$ gives the claim.  
\end{proof}

\noindent\emph{Proof of Theorem \ref{thmBl} (completed).} The energy identity \eqref{enBla} follows directly from \cite[Lemma 3.1]{BMRS} or \cite[Lemma 5.1]{HM}. Moreover, traces of Blaschke products $\tilde h$ of degree $d$ are the only maps in {$\dot H^\frac12(\R,\sph^1)$ of degree $d$} with energy $2\pi |d|$ {(still by \cite[Lemma 3.1]{BMRS} or \cite[Lemma 5.1]{HM})}. { Since \eqref{enBla2} follows from Lemma \ref{l1}, it follows at once that the infimum in \eqref{enBla2} is not attained.} \hfill$\square$

\medskip

\subsection{Proof of Theorem \ref{Bl3}}

We start with a few lemmas. First we recall Lemma 4.1 from \cite{HM}.

\begin{lem}[\cite{HM}]\label{l2} Let $u\in \M E_g(I)$ be such that for some
$x_0\in I$ we have $u'(x_0)\neq 0$ and
\begin{equation}\label{perp}
\partial_y \tilde u(x_0,0)\perp T_{u(x_0)}\sph^1,
\end{equation}
where $\tilde u$ denotes the Poisson harmonic extension of $u$. Then there exists
$v\in\M E_g$ such that
\[
E(v)<E(u)+2\pi,\qquad
\deg(v)=\deg(u)-\operatorname{sign}(u(x_0)\wedge u'(x_0)).
\]
\end{lem}

\begin{lem}\label{l4} Let $(u_n)$ be a minimizing sequence for $m_{g,d}$ in
$\M E_{g,d}(I)$, and assume that, up to a subsequence,
$u_n\rightharpoonup u\in \M E_{g,\ell}(I)$ with $\ell\neq d$. Then
\[
m_{g,\ell}=m_{g,d}-2\pi|d-\ell|,
\qquad
E(u)=m_{g,\ell}.
\]
\end{lem}

\begin{proof}
Proposition \ref{p:jump} gives
\[
E(u)\le \liminf_{n\to\infty}E(u_n)-2\pi|d-\ell|
       =m_{g,d}-2\pi|d-\ell|.
\]
Together with Lemma~\ref{l1},
\[
m_{g,\ell}\le E(u)
\le m_{g,d}-2\pi|d-\ell|
\le m_{g,\ell}.
\]
Thus all inequalities are equalities.
\end{proof}

\begin{lem}\label{lem:strict-comparison}
Let $g\in C^1(\R,\sph^1)\cap \dot H^{\frac12}(\R,\sph^1)$. Assume that
\[
q:=\frac1{2\pi}\int_{\R\setminus I} g\wedge g'\,dx\in \R.
\]
Given $d\in \mathbb{Z}$, assume that $m_{g,d}$ is achieved by a map $u_d$. Then:
\begin{itemize}
\item if $d>q$, then
\[
m_{g,\ell}<m_{g,d}+2\pi(d-\ell)\qquad\text{for every }\ell<d;
\]
\item if $d<q$, then
\[
m_{g,\ell}<m_{g,d}+2\pi(\ell-d)\qquad\text{for every }\ell>d;
\]
\item if $d=q$ and $u_d$ is non-constant in $I$  then 
\[
m_{g,\ell}<m_{g,d}+2\pi|d-\ell|
\qquad\text{for every }\ell\neq d.
\]
\end{itemize}
\end{lem}

\begin{proof}
Since $u_d$ is half-harmonic in $I$ and $g\in C^1(\R,\sph^1)$, by the regularity theory (\cite{DLR,Sch} and Lemma \ref{bdryreg}), $u_d\in C^\infty(I)\cap C^0(\bar I)$. Since $u_d=g$ on $\R\setminus I$ and $\deg(u_d)=d$, choosing a continuous lifting $\varphi\in C^\infty(I)\cap C^0(\bar I)$ of $u_d$ on $\bar I$, i.e. $u_d=e^{i\varphi}$ on $\bar I$, we have
\[
\frac1{2\pi}\int_I u_d\wedge u_d'\,dx=d-q,
\]
where the integral is understood as
\[
\int_I u_d\wedge u_d'\,dx
:=\lim_{\varepsilon\downarrow0}
\int_{-1+\varepsilon}^{1-\varepsilon} u_d\wedge u_d'\,dx
=\lim_{\ve\downarrow 0}\varphi(1-\ve)-\varphi(-1+\ve)= \varphi(1)-\varphi(-1).
\]
If $d>q$, then $u_d(x_0)\wedge u_d'(x_0)>0$ for some $x_0\in I$. Since $u_d$ is half-harmonic, \eqref{perp} is satisfied, hence Lemma~\ref{l2} gives
\[
m_{g,d-1}<m_{g,d}+2\pi .
\]
For $\ell<d-1$, Lemma~\ref{l1} then yields
\[
m_{g,\ell}\le m_{g,d-1}+2\pi(d-1-\ell)
          < m_{g,d}+2\pi(d-\ell).
\]
The case $d<q$ is identical, using a point $x_0$ where $u_d(x_0)\wedge u_d'(x_0)<0$.

Finally assume $d=q$. Then
\[
\int_I u_d\wedge u_d'\,dx=0,
\]
and since $u_d$ is not constant in $I$, the function $u_d\wedge u_d'$ takes both positive and negative values in $I$. Applying Lemma~\ref{l2} in both directions gives the strict inequalities for $\ell=d\pm1$, and Lemma~\ref{l1} gives all remaining $\ell$.
\end{proof}

\begin{lem}\label{lem:attainment-criterion}
Assume that for a fixed $d$ one has
\[
m_{g,d}<m_{g,\ell}+2\pi|d-\ell|
\qquad\text{for every }\ell\neq d.
\]
Then $m_{g,d}$ is achieved.
\end{lem}

\begin{proof}
Let $(u_n)\subset\M E_{g,d}(I)$ be a minimizing sequence and, up to a subsequence, let
$u_n\rightharpoonup u\in \M E_{g,\ell}(I)$. If $\ell=d$, then $u$ achieves $m_{g,d}$ by lower semicontinuity (Proposition \ref{p:jump}). If $\ell\neq d$, Lemma~\ref{l4} gives
\[
m_{g,d}=m_{g,\ell}+2\pi|d-\ell|,
\]
contradicting the assumption.
\end{proof}

\noindent\emph{Proof of Theorem \ref{Bl3} (completed).}
By Theorem~\ref{thmBl} $m_{g,k}=2\pi k$ is achieved by $g$, and $m_{g,d}$ is not achieved for any $d>k$. Moreover, if a minimizing sequence for some $m_{g,d}$ weakly converges to $u\in \mathcal{E}_{g,r}$ with $r\neq d$, then Lemma~\ref{l4} gives
\[
 m_{g,d}=m_{g,r}+2\pi |d-r|,
\]
and $m_{g,r}$ is achieved. We shall use this fact repeatedly.

We also observe that if $q\in\mathbb Z$ and $m_{g,q}$ is achieved by a minimizer $u_q$, then  $u_q$ is non-constant in $I$ (see e.g. \cite[Lemma 4.3]{HM}).

\medskip

\noindent\emph{The non-integer case.}
Assume that $q\in(\ell,\ell+1)$ for some $\ell\in\{0,\dots,k-1\}$. We prove that $m_{g,\ell}$ and $m_{g,\ell+1}$ are achieved; if $\ell+1=k$, the latter is already known from Theorem~\ref{thmBl}.

First assume $\ell+1<k$. Since $m_{g,k}$ is achieved and $k>q$, Lemma~\ref{lem:strict-comparison} gives
\begin{equation}\label{strictell}
 m_{g,\ell+1}<m_{g,k}+2\pi(k-\ell-1).
\end{equation}
Let $(u_n)\subset \mathcal{E}_{g,\ell+1}$ be a minimizing sequence for $m_{g,\ell+1}$ and assume that $u_n\rightharpoonup u\in \mathcal{E}_{g,r}$ weakly for some $r\neq \ell+1$. Then
\[
 m_{g,\ell+1}=m_{g,r}+2\pi|\ell+1-r|,
\]
and $m_{g,r}$ is achieved by $u$. If $r>k$, this contradicts Theorem~\ref{thmBl}. If $r=k$, it contradicts \eqref{strictell}. If $\ell+1<r<k$, then $r>q$, and Lemma~\ref{lem:strict-comparison} gives
\[
 m_{g,\ell+1}<m_{g,r}+2\pi(r-\ell-1),
\]
again a contradiction. Finally, if $r\le \ell$, then $r<q$, and Lemma~\ref{lem:strict-comparison} gives
\[
 m_{g,\ell+1}<m_{g,r}+2\pi(\ell+1-r),
\]
contradiction. Hence $m_{g,\ell+1}$ is achieved.

Now $m_{g,\ell+1}$ is achieved and, since $\ell+1>q$, Lemma~\ref{lem:strict-comparison} gives
\[
 m_{g,\ell}<m_{g,\ell+1}+2\pi.
\]
Consider now a sequence $(v_n)\subset \mathcal{E}_{g,\ell}$ minimizing for $m_{g,\ell}$, and assume that $v_n\rightharpoonup v\in \mathcal{E}_{g,r}$ weakly for some $r\neq \ell$. Then $m_{g,r}$ is achieved and
\begin{equation}\label{mgell}
 m_{g,\ell}=m_{g,r}+2\pi|\ell-r|.
\end{equation}
The case $r>k$ is impossible by Theorem~\ref{thmBl}. If $r\ge \ell+1$, then $r>q$, and Lemma~\ref{lem:strict-comparison} gives
\[
 m_{g,\ell}<m_{g,r}+2\pi(r-\ell),
\]
contradiction. If $r<\ell$, then $r<q$, and Lemma~\ref{lem:strict-comparison} gives
\[
 m_{g,\ell}<m_{g,r}+2\pi(\ell-r),
\]
again a contradiction. Therefore $m_{g,\ell}$ is achieved.

Taking $\ell=k-1$ gives the first item of the theorem. Taking $\ell\in\{0,\dots,k-2\}$ gives the fourth item; in particular, for $q\in(k-2,k-1)$ it gives attainment in the classes $k-2,k-1,k$.

\smallskip

\noindent\emph{The integer case.}
Assume now that $q=\ell\in\{1,\dots,k-1\}$. We prove that $m_{g,\ell+1}$, $m_{g,\ell}$ and $m_{g,\ell-1}$ are achieved, with the convention that $m_{g,\ell+1}=m_{g,k}$ is already achieved if $\ell=k-1$.

If $\ell+1<k$, the proof that $m_{g,\ell+1}$ is achieved is the same as above, except for the possible weak limit $r=\ell=q$. In that case Lemma~\ref{l4} makes $m_{g,\ell}$ achieved; by the non-constancy observation at the beginning of the proof, Lemma~\ref{lem:strict-comparison} applies at level $\ell=q$ and gives
\[
 m_{g,\ell+1}<m_{g,\ell}+2\pi,
\]
contradicting the equality given by Lemma~\ref{l4}. Thus $m_{g,\ell+1}$ is achieved.

Since $\ell+1>q$ and $m_{g,\ell+1}$ is achieved, Lemma~\ref{lem:strict-comparison} gives
\[
 m_{g,\ell}<m_{g,\ell+1}+2\pi.
\]
As before, if a minimizing sequence for $m_{g,\ell}$ weakly converges to a function of degree $r\neq \ell$, Lemma~\ref{l4} makes $m_{g,r}$ achieved and gives equality \eqref{mgell}. The case $r>k$ again is impossible. If $r\ge \ell+1$, then $r>q$ and Lemma~\ref{lem:strict-comparison} gives the opposite strict inequality. If $r<\ell$, then $r<q$ and Lemma~\ref{lem:strict-comparison} again gives the opposite strict inequality. Therefore $m_{g,\ell}$ is achieved.

The minimizer at level $\ell=q$ is non-constant in $I$, hence Lemma~\ref{lem:strict-comparison} gives
\[
 m_{g,\ell-1}<m_{g,\ell}+2\pi.
\]
If a minimizing sequence for $m_{g,\ell-1}$  weakly converges to a map of degree $r\neq \ell-1$, then by Lemma~\ref{l4} $m_{g,r}= m_{g,\ell-1}-2\pi |\ell-1-r|$ is achieved. The case $r>k$ is impossible. If $r\ge \ell$, Lemma~\ref{lem:strict-comparison} gives a strict comparison downward to $\ell-1$; for $r=\ell$ we use again the non-constancy of the minimizer at level $q$. If $r<\ell-1$, then $r<q$ and Lemma~\ref{lem:strict-comparison} gives a strict comparison upward to $\ell-1$. In all cases we contradict the equality from Lemma~\ref{l4}. Thus $m_{g,\ell-1}$ is achieved.

For $\ell=k-1$ this gives the second item of the theorem with $q=k-1$. For $\ell\in\{1,\dots,k-2\}$ it gives the third item.
\hfill$\square$

\appendix

\section{Appendix}

We give a simple proof of the boundary regularity result that we used in this paper. A more general result can be found in \cite{Sch}.

\begin{lem}\label{bdryreg} Let $g\in C^0\cap\dot H^\frac12(\R,\sph^1)$ and let $u\in \dot H^{\frac12}(\R,\sph^1)$ be $1/2$-harmonic in an interval $(a,b)$ with $u=g$ in $\R\setminus (a,b)$. Then $u$ is continuous in $\R$.
\end{lem}

\begin{proof}
Up to scaling, translation and reflection, it is enough to consider $(a,b)=(0,1)$ and prove continuity at $0$.
Interior regularity gives $u\in C^\infty((0,1),\sph^1)$.

Assume by contradiction that $u(x)\not\to g(0)$ as $x\downarrow0$. Then there are
$r_k\downarrow0$ and $\eta>0$ such that
\[
|u(r_k)-g(0)|\ge \eta .
\]
Set
\[
u_k(x):=u(r_kx).
\]
Then $u_k$ is $1/2$-harmonic in $(0,r_k^{-1})$, and $u_k(x)=g(r_kx)$
for $x\le0$. Since $g$ is continuous, we have
$u_k\to g(0)$ locally uniformly in $(-\infty,0]$.

We claim that for every $R>0$,
\[
[u_k]^2_{\dot H^{\frac12}((-R,R))}
=
\int_{-R}^{R}\int_{-R}^{R}
\frac{|u_k(x)-u_k(y)|^2}{|x-y|^2}\,dx\,dy
\to0,\quad \text{as }k\to+\infty.
\]
Indeed, by the change of variables $X=r_kx$, $Y=r_ky$,
\[
[u_k]^2_{\dot H^{\frac12}((-R,R))}
=
\int_{-Rr_k}^{Rr_k}\int_{-Rr_k}^{Rr_k}
\frac{|u(X)-u(Y)|^2}{|X-Y|^2}\,dX\,dY\to 0,\quad \text{as }k\to+\infty
\]
by absolute continuity of the double integral defining
$[u]_{\dot H^{\frac12}(\mathbb R)}$.

By the fractional Poincar\'e inequality on $(-R,R)$, setting $(u_k)_{(-R,R)}=\fint_{-R}^R u_k dx$, we have
\[
\|u_k-(u_k)_{(-R,R)}\|_{L^2((-R,R))}\to 0, \quad \text{as }k\to+\infty.
\]
Since $u_k\to g(0)$ uniformly on $(-R,-R/2)$ as $k\to +\infty$, we have
\[
(u_k)_{(-R,R)}\to g(0), \quad \text{as }k\to+\infty,
\]
hence
\[
u_k\to g(0)
\quad\text{in }L^2((-R,R)),
\]
as $k\to+\infty$, for every $R>0$.

Now fix $0<s<t<\infty$. For $k$ large, $(s/2,2t)\Subset (0,r_k^{-1})$, and $u_k$
is $1/2$-harmonic there. Moreover
\[
\lim_{k\to\infty}[u_k]_{\dot H^{\frac12}((s/2,2t))}=0.
\]
By the interior $\varepsilon$-regularity for $1/2$-harmonic maps (see e.g. \cite[Theorem 1.1]{Sch2}), 
the sequence $(u_k)$ is uniformly bounded in $C^\alpha([s,t])$
for some $\alpha>0$. Since $u_k\to g(0)$ in $L^2((s/2,2t))$, the equicontinuity implies
\[
u_k\to g(0)
\quad\text{uniformly on }[s,t]\text{ as }k\to+\infty.
\]
Taking $s=1/2$, $t=3/2$, we obtain
\[
u(r_k)=u_k(1)\to g(0),\quad  \text{as }k\to+\infty,
\]
contradicting $|u(r_k)-g(0)|\ge\eta$. Thus $u(x)\to g(0)$ as $x\downarrow0$.
\end{proof}


\end{document}